\documentclass[11pt]{amsart}
\usepackage[utf8x]{inputenc}
\usepackage[english]{babel}
\usepackage[T1]{fontenc}
 
\usepackage{graphicx}
\usepackage{caption}
\usepackage{subcaption}

\usepackage{amssymb,amsmath}
\usepackage{mathtools}

\usepackage{hyperref}
\usepackage[capitalise]{cleveref}

\usepackage{indentfirst}
\usepackage{enumerate,amsmath,amssymb, mathrsfs,mathtools}
\usepackage{appendix}
\usepackage{latexsym}
\usepackage{url}
\usepackage{color}
\usepackage{accents}
\usepackage{setspace}
\usepackage{pdfpages}
\usepackage{stmaryrd}

\usepackage{amsrefs}

\usepackage[margin=2.5cm]{geometry}
\allowdisplaybreaks

\BibSpec{article}{%
	+{}{\PrintAuthors} {author}
	+{,}{ } {title}
	+{, }{\textit } {journal}
	+{}{ \parenthesize} {date}
		+{,  }{no. } {volume}
	+{,}{ } {pages}
	+{,}{ } {note}
}
\usepackage{bbm}

\ExplSyntaxOn

\NewExpandableDocumentCommand{\gobblefirst}{m}
{
	\tl_tail:n { #1 }
}

\ExplSyntaxOff
\BibSpec{book}{%
	+{}{\PrintAuthors}  {author}
	+{. }{}{title}
	+{,}{ }{series}
	+{,}{ vol.~}{volume}
	+{. }{\textit}{publisher}
	+{,}{ ISBN \gobblefirst}{isbn}
}
\BibSpec{collection.article}{%
	+{}{\PrintAuthors}{author}
	+{, }{}{title}
	+{, }{\textit}{booktitle}
	+{, }{ \DashPages}{pages}
	+{,}{ }{series}
	+{, }{}{volume}
	+{, }{\textit}{publisher}
	+{,}{ }{date}
}

\mathcode`l="8000
\begingroup
\makeatletter
\lccode`\~=`\l
\DeclareMathSymbol{\lsb@l}{\mathalpha}{letters}{`l}
\lowercase{\gdef~{\ifnum\the\mathgroup=\m@ne \ell \else \lsb@l \fi}}%
\endgroup

\def\XXint#1#2#3{{\setbox0=\hbox{$#1{#2#3}{\int}$ }
		\vcenter{\hbox{$#2#3$ }}\kern-.6\wd0}}

\newtheorem{proposition}{Proposition}

\newtheorem{definition}[proposition]{Definition}
\newtheorem{theorem}[proposition]{Theorem}
\newtheorem{lemma}[proposition]{Lemma}
\newtheorem{corollary}[proposition]{Corollary}
\newtheorem{remark}[proposition]{Remark}

\title[Parallel (co-)tractors and the geometry of first BGG solutions on AG-structures]{Parallel (co-)tractors and the geometry of first BGG solutions on almost Grassmannian structures}
\author{Zhangwen Guo}
\thanks{The author is grateful to her doctoral advisor A. \v Cap for introducing her to the problems solved in this paper and his steady support. This research was funded in part by the Austrian Science Fund (FWF): 10.55776/P33559 and 10.55776/Y963. For open access purposes, the author has applied a CC BY public copyright license to any author-accepted manuscript version arising from this submission.}
\address{\textnormal{Zhangwen Guo  \newline \indent
		University of Vienna \newline \indent
		Faculty of Mathematics  \newline \indent
		Oskar-Morgenstern-Platz 1 \newline \indent 1090 Vienna,	Austria %\newline\indent 		 \href{https://orcid.org/0000-0003-1676-0824}{https://orcid.org/0000-0003-1676-0824}
        \newline \indent
		  \href{mailto:zhangwen.guo@univie.ac.atm}{zhangwen.guo@univie.ac.at}}
}

\begin{document}
%\date{\today}\onehalfspacing
\begin{abstract}
We study the standard tractor bundle and the standard cotractor bundle of an almost Grassmannian structure: We provide explicit formulae for their splitting operators, first BGG operators as well as prolongation connections. We characterize parallel tractors and cotractors as well as the solutions of the BGG operators in standard geometric terms. Moreover, we describe the geometry canonically endowed on the zero locus of a solution of the first BGG operators.
\end{abstract}
	\maketitle 

\section{Our results}\label{1}

Assume that $n>2$. An almost Grassmannian structure (AG-structure) of type $(2,n)$ on a $2\,n$-dimensional manifold $M$ is defined by two auxiliary vector bundles $E\to M$ and $F\to M$ of rank $2$ and of rank $n$, respectively, and isomorphisms
\begin{align}\label{AG isomorphisms}
    E^*\otimes F\cong TM\quad\text{and}\quad\wedge^2E^*\cong\wedge^n\,F.
\end{align}
We denote this structure by $(M,E,F)$. Every AG-structure $(M,E,F)$ comes with a collection of so-called Weyl connections. These are precisely those pairs of connections on $(E,F)$ which are compatible with the second isomorphism in \eqref{AG isomorphisms} and, via the first isomorphism in \eqref{AG isomorphisms}, induce a connection on $TM$ whose torsion $\tau\in\Omega^2(M,TM)$ satisfies specific properties and is, in particular, totally trace-free; see \eqref{harmonic components AG}. There is a canonical contraction of $\tau\otimes\tau$ resulting in a $(0,2)$-tensor field; see \eqref{Weyl tensor Ricci type contraction}. This contraction is clearly symmetric. We will denote it by
\begin{align}\label{torsion contraction def}
    tr(\iota_\tau\tau)\in\Gamma(\mathrm{Sym}^2\,T^*M).
\end{align}

On the other hand, every AG-structure $(M,E,F)$ of type $(2,n)$ carries a so-called standard tractor bundle $\mathcal T$, obtained for example from the canonical Cartan geometry determined by the AG-structure. It is a vector bundle on $M$ coming with a canonical connection as well as a short exact sequence
\begin{align*}
    0\to E\to\mathcal T\to F\to 0.
\end{align*}
The general theory developed in \cite{CSS} gives rise to a so-called first splitting operator 
\[L:\Gamma(F)\to\Gamma(\mathcal T)\]
and a so-called first Bernstein-Gelfand-Gelfand (BGG) operator 
\[D:\Gamma(F)\to\Omega^1(M,F).\]
If a section of $\mathcal T$ is parallel, it is necessarily of the form $L\,\eta$ for some $\eta\in\Gamma(F)$ such that $D\,\eta=0$. In the current paper, we obtain the explicit expression of these operators in \cref{splitting operator} and prove in \cref{parallel tractor} that the space of parallel sections of $\mathcal T$ is precisely
\[\{L\,\eta:\eta\in\Gamma(F),D\,\eta=0\text{ and }tr(\iota_\tau\tau)(\eta)=0\}\subseteq\Gamma(\mathcal T)\]
where we have viewed
\[tr(\iota_\tau\tau)\in\Gamma(F^*\otimes (E\otimes E\otimes F^*)).\]
We then focus on studying the solutions
\[\{\eta\in\Gamma(F):D\,\eta=0\}\]
of the first BGG operator associated to $\mathcal T$. In \cref{prolongation connection} we construct a prolongation connection $\hat\nabla^{\mathcal T}$ on $\mathcal T$ such that
\begin{align*}
D\,\eta=0\quad\iff\quad \hat\nabla^{\mathcal T}(L\,\eta)=0        
\end{align*}
for every $\eta\in\Gamma(F)$. There, we also describe a Cartan geometry $(\mathcal G\to M,\hat\omega)$ whose underlying geometry is $(M,E,F)$ and such that the prolongation connection $\hat\nabla^{\mathcal T}$ is the tractor connection canonically derived from the Cartan connection $\hat\omega$. For convenience, we develop the rest of the results assuming in addition that $M$ is connected. Let $0\neq\eta\in\Gamma(F)$ such that $D\,\eta=0$. We prove in \cref{open dense} that
\[\{x\in M:\eta(x)\neq 0\}\]
is an open dense subset of $M$. We prove in \cref{1st BGG operator} that if $\eta\in\Gamma(F)$ is nowhere vanishing, then $D\,\eta=0$ if and only if $\nabla\eta=0$ for some Weyl connection $\nabla$. Now consider the subset
\[N=\{x\in M:\eta(x)=0\}\subseteq M.\]
We prove in \cref{zero locus} that $N$ is an embedded submanifold of $M$ of codimension $n$. 
We also prove in \cref{zero locus} that for every $x\in N$, there holds
\[T_xN=\ell_x\otimes F_x\subseteq T_xM\]
where $\ell_x\in E_x^*$ is the one-dimensional subspace annihilating $(L\,\eta)(x)\in E_x\setminus\{0\}$. Moreover, we prove in \cref{canonical connections} that the zero locus $N$ comes with a canonical projective structure.

All results above has a corresponding version for the standard cotractor bundle $\mathcal T^*$ presented in the same locations as mentioned. Here, the zero locus $\tilde N$ for a solution of the first BGG operator is a $2\,(n-1)$-dimensional embedded submanifold of $M$ carrying an AG-structure of type $(2,n-1)$.

\section{Parabolic geometries and Weyl structures}\label{2}
In this section, we recall that every AG-structure $(M,E,F)$ of type $(2,n)$ comes with a canonical parabolic geometry, from which we may derive the fundamental invariants as well as all Weyl objects by certain procedures according to a general theory. We refer the reader to, e.g., \cite{Book} for detailed background.

Let $G$ be a Lie group with Lie algebra $\mathfrak g$ and $P$ a closed Lie subgroup of $G$ with Lie algebra $\mathfrak p$.

\begin{definition}\label[definition]{Cartan geomtry definition}
    A Cartan geometry $(\mathcal G\to M,\omega)$ of type $(G,P)$ is a principal $P$-bundle $\mathcal G\to M$ together with a Cartan connection $\omega\in\Omega^1(\mathcal G,\mathfrak g)$, i.e., a $\mathfrak g$-valued one-form on $\mathcal G$ such that (i) $\omega$ is $P$-equivariant with respect to the principal right action and the adjoint action of $G$ on $\mathfrak g$ restricted to $P\subseteq G$, (ii) $\omega(\zeta_X)=X$ where $\zeta_X\in\mathfrak X(\mathcal G)$ is the fundamental vector field generated by $X\in\mathfrak p$ and (iii) $T\mathcal G\xrightarrow[\cong]{(\pi,\omega)}\mathcal G\times\mathfrak g$ where $\pi:T\mathcal G\twoheadrightarrow\mathcal G$ is the natural projection.

    The curvature of the Cartan connection is the two-form $\kappa\in\Omega^2(\mathcal G,\mathfrak g)$ given by $\kappa(\xi,\eta)=d\omega(\xi,\eta)+[\omega(\xi),\omega(\eta)]$ for all $\xi,\eta\in\mathfrak X(\mathcal G)$.
\end{definition}
Note that the curvature $\kappa\in\Omega^2(\mathcal G,\mathfrak g)$ is $P$-equivariant and horizontal.
In particular, using the isomorphism 
\[TM\cong\mathcal G\times_P\mathfrak g/\mathfrak p\] 
induced by the Cartan connection $\omega:T\mathcal G\to\mathfrak g$, the Cartan curvature can be expressed as an $P$-equivariant map $\kappa:\mathcal G\to L(\wedge^2\,\mathfrak g/\mathfrak p,\mathfrak g)$ and as a two-form $\kappa\in\Omega^2(M,\mathcal G\times_P\mathfrak g)$. 

Now let
\begin{align*}
P=&\left(\begin{array}{c|ccc}
	*&&*&\\
 \hline&&&\\
 0&&*&\\
 &&&
\end{array}\right)\subseteq G=SL(n+2,\mathbb R),\\
G_0=&\left(\begin{array}{c|ccc}
	*&&0&\\
 \hline&&&\\
 0&&*&\\
 &&&
\end{array}\right)\subseteq P\quad\text{and}\quad
P_+=\left(\begin{array}{c|ccc}
	\mathbb I_2&&*&\\
 \hline&&&\\
 0&&\mathbb I_n&\\
 &&&
\end{array}\right)\subseteq P.
\end{align*}
The block size is $(2,n)\times(2,n)$. Additionally, the Lie algebra of $SL(n+2,\mathbb R)$ is endowed with a $|1|$-grading
\begin{align}\label{gradingAG}
\mathfrak{sl}(n+2,\mathbb R)=\mathfrak g=\left(\begin{array}{c|ccc}
	\mathfrak g_0&&\mathfrak g_1&\\
 \hline&&&\\
 \mathfrak g_{-1}&&\mathfrak g_0&\\
 &&&
\end{array}\right).\end{align}
We also use the following notation.
\begin{align*}
    \mathfrak g_-=\mathfrak g_{-1},\quad
    \mathfrak p=\mathfrak g_0\oplus\mathfrak g_1
    \quad\text{and}\quad
    \mathfrak g_+=\mathfrak p_+=\mathfrak g_{1}.
\end{align*}
Thus $\mathfrak p,\mathfrak g_0$ and $\mathfrak p_+$ are Lie algebras of $P,G_0$ and $P_+$, respectively. Since $P$ is a so-called parabolic subgroup of $SL(n+2,\mathbb R)$, every Cartan geometry of type $(SL(n+2,\mathbb R),P)$ is called a parabolic geometry of type $(SL(n+2,\mathbb R),P)$.

The Killing form on $\mathfrak g$ induces a $P$-invariant identification \[(\mathfrak g/\mathfrak p)^*\cong\mathfrak p_+.\]
Any $G$-representation $\mathbb V$ gives rise to a complex of $P$-modules with differentials
\[\partial^*_\mathbb V:\wedge^{k+1}\,\mathfrak p_+\otimes\mathbb V\to\wedge^{k}\,\mathfrak p_+\otimes\mathbb V\] 
given by
\begin{align}\label{codiff}
    \partial^*_\mathbb V(Z_0\wedge\dots\wedge Z_k\otimes v)&=\sum_{i=0}^k(-1)^{i+1}\,Z_0\wedge\dots\wedge\hat{Z_i}\wedge\dots\wedge Z_k\otimes(Z_i\,v)
\end{align}
for all $Z_0,\dots, Z_n\in\mathfrak p_+$ and $v\in\mathbb V$. Here, the hat indicates that the object is omitted.

A parabolic geometry $(\mathcal G\to M,\omega)$ of type $(SL(n+2,\mathbb R),P)$ with curvature 
\[\kappa:\mathcal G\to\wedge^2\mathfrak p_+\otimes\mathfrak g\]
is said to be normal if and only if $\partial^*_{\mathfrak g}\circ\kappa=0$.

We recall the following categorical correspondence which follows from \cite{Book}*{Subsection 4.1.3 and Theorem 3.1.14}.

\begin{lemma}\label[lemma]{AG parabolic}
Every parabolic geometry $(\mathcal G\to M,\omega)$ of type $(SL(n+2,\mathbb R),P)$ has an underlying AG-structure \[(M,E=\mathcal G\times_{P}\mathbb R^2,F=\mathcal G\times_{P}\mathbb R^{n+2}/\mathbb R^2)\]
of type $(2,n)$.

Conversely, for any AG-structure $(M,E,F)$ of type $(2,n)$, there is, up to isomorphisms, a unique normal parabolic geometry of type $(SL(n+2,\mathbb R),P)$ whose underlying geometry is $(M,E,F)$.
\end{lemma}

Let $(\mathcal G\to M,\omega)$ be the normal parabolic geometry of type $(SL(n+2,\mathbb R),P)$ canonically associated to a given AG-structure $(M,E,F)$. Then the curvature of the Cartan connection $\omega$ descends to the so-called harmonic curvature 
\[\kappa_H:\mathcal G\to\ker\,\partial^*_\mathfrak g/\mathrm{im}\,\partial^*_\mathfrak g.\]
For $n>2$, $\kappa_H$ decomposes into 
\begin{equation}\label{harmonic components AG}
\begin{aligned}
    \tau&\in\Gamma\left((\mathrm{Sym}^2\,E\otimes E^*)_o\otimes(\wedge^2\,F^*\otimes F)_o\right)\subseteq\Omega^2(M,TM)\quad\text{and}\\
    \rho&\in\Gamma(\wedge^2\,E\otimes \mathrm{Sym}^2\,F^*\otimes\mathfrak {sl}(F))\subseteq\Omega^2(M,\mathfrak {sl}(F)).
\end{aligned}
\end{equation}
Here, the natural decomposition
\[\wedge^2\,T^*M=\wedge^2\,(E\otimes F^*)=(\mathrm{Sym}^2\,E\otimes\wedge^2\,F^*)\oplus(\wedge^2\,E\otimes \mathrm{Sym}^2\,F^*)\]
is used and
\begin{align*}
(\mathrm{Sym}^2\,E\otimes E^*)_o\subseteq \mathrm{Sym}^2\,E\otimes E^*
\quad\text{and}\quad
(\wedge^2\,F^*\otimes F)_o\subseteq \wedge^2\,F^*\otimes F
\end{align*}
denote the trace-free component. Stretching notation, we call $\tau$ the harmonic torsion, or, simply, the torsion, of $(M,E,F)$ and $\rho$ the harmonic curvature of $(M,E,F)$. 

Let $(\mathcal G\to M,\omega)$ be the normal parabolic geometry canonically associated to an AG-structure $(M,E,F)$ of type $(2,n)$. Since $P/P_+\cong G_0$, $\mathcal G_0=\mathcal G/P_+$ is a principal $G_0$-bundle. Note that
\begin{align}\label{frame data}
    E\cong\mathcal G_0\times_{G_0}\mathbb R^2\quad\text{and}\quad
    F\cong\mathcal G_0\times_{G_0}\mathbb R^n
\end{align}
encode precisely every local trivializations of $(E,F)$ preserving the second defining isomorphism in \eqref{AG isomorphisms}. This reflects the fact that the Cartan connection $\omega$ descends to a soldering form $\theta\in\Omega^1(\mathcal G_0,\mathfrak g_{-1})$ and thus $(\mathcal G_0,\theta)$ is the $G_0$-structure which equivalently encodes $(M,E,F)$. We refer the reader to, e.g., \cite{Book}*{Proposition 3.1.15 (2)} for details.

Consider the natural projection $\mathcal G\twoheadrightarrow\mathcal G_0$. It admit $G_0$-equivariant sections
\[\sigma:\mathcal G_0\to\mathcal G,\]
called the Weyl structures of $(M,E,F)$.
Since the block decomposition of $\mathfrak g$ is $G_0$-invariant, we may interpret the pullback of the Cartan connection $\omega$ and its curvature $\kappa$ in each block component as tensors on $(M,E,F)$. It is easy to see that, $\sigma^*\omega_{\mathfrak g_{-1}}$ and $\sigma^*\kappa_{\mathfrak g_{-1}}$ are independent on the choice of a Weyl structure. We obtain the following Weyl objects associated to the Weyl structure $\sigma$.

\begin{equation}\label{induced Weyl objects}
   \begin{aligned}
   % \begin{drcases}
\sigma^*\omega_{\mathfrak g_{0}}&\in\Omega^1(\mathcal G_0,\mathfrak g_{0})^{G_0}\text{ (the Weyl connection)}\\
\mathrm P=\sigma^*\omega_{\mathfrak g_{1}}&\in
\Omega^1(M,T^*M)\text{ (the Rho tensor)}\\
(W,W')=\sigma^*\kappa_{\mathfrak g_{0}}&\in
\Omega^2(M,\mathfrak s(L(E,E)\oplus L(F,F)))\text{ (the Weyl tensor)}\\
Y=\sigma^*\kappa_{\mathfrak g_{1}}&\in
\Omega^2(M,T^*M)\text{ (the Cotton-York tensor)}.
%    \end{drcases}
\end{aligned}
\end{equation}
where
\begin{align}\label{frak s}
   \mathfrak s(L(E,E)\oplus L(F,F)))=\Big(L(E,E)\oplus L(F,F)\Big)\cap\mathfrak{sl}(E\oplus F).
\end{align}
The Weyl connection $\sigma^*\omega_{\mathfrak g_{0}}$ is a principal connection on $\mathcal G_0$, thus induces a pair of connections
\[(\nabla^E,\nabla^F)\]
on $(E,F)$. The Weyl tensor $(W,W')$ is decomposed into
\begin{align*}
    W\in\Omega^2(M,L(E,E))\quad\text{and}\quad W'\in\Omega^2(M,L(F,F)).
\end{align*}
Note that the principal connections on $\mathcal G_0$ are equivalent to all pairs $(\nabla^E,\nabla^F)$ of connections on $(E,F)$ compatible with the defining isomorphism $\wedge^2\, E^*\cong \wedge^n\,F$. In fact, it follows from, e.g., \cite{Book}*{Theorems 3.1.12 and 5.2.3} that $\sigma^*\kappa_{\mathfrak g_{-1}}$, the harmonic torsion $\tau$ and the torsion of any Weyl connection on $TM$ are the same section of $\Omega^2(M,TM)$. Moreover, since $(\mathcal G\to M,\omega)$ is $|1|$-graded, there is a one-to-one correspondence between Weyl structures and Weyl connections of $(\mathcal G\to M,\omega)$; see, e.g., \cite{Book}*{Proposition 5.1.1 and Proposition 5.1.6}. 
This is how the characterization of Weyl connections described at the beginning of \cref{1} is obtained.
\section{Weyl tensor and \texorpdfstring{$tr(\iota_\tau\tau)$}{}}
Let $(\mathcal G\to M,\omega)$ be a parabolic geometry of type $(G,P)$. Consider the canonical extension $j:\mathcal G\to \mathcal G\times_PG$ of $\mathcal G$ to a principal $G$-bundle. There is a unique principal connection on $\mathcal G\times_PG$ whose pullback under $j$ coincides with $\omega$; see, e.g., \cite{Book}*{Theorem 1.5.6}. In particular, every $G$-representation $\mathbb V$ gives rise to a so-called tractor bundle 
\[\mathcal V=(\mathcal G\times_PG)\times_G\mathbb V=\mathcal G\times_P\mathbb V\]
whose tractor connection $\nabla^\mathcal V$ is induced by the principal connection mentioned above. It turns out that $\nabla^\mathcal V_\xi s\in\Gamma(\mathcal V)$ is equivalent to the $P$-equivariant map
\begin{align}\label{tractor connection general formula}
df_s(\Tilde{\xi})+\omega(\Tilde{\xi})\,f_s:\mathcal G\to\mathbb V.
\end{align}
Here, $\xi\in\mathfrak X(M)$ is lifted to $\Tilde{\xi}\in\mathfrak X(\mathcal G)$ and $s\in\Gamma(\mathcal V)$ is equivalent to a $P$-equivariant map $f_s:\mathcal G\to\mathbb V$.

Assume in addition that $(\mathcal G\to M,\omega)$ is normal. Given any Weyl structure $\sigma:\mathcal G_0\to\mathcal G$, we may express the tractor connection $\nabla^\mathcal V$ in terms of the following Weyl objects associated to $\sigma$: the Weyl connection $\nabla$ on $\mathcal V\cong\mathcal G_0\times_{G_0}\mathbb V$, the Rho-tensor $\mathrm P\in\Omega^1(M,T^*M)$ and
\begin{align*}
\bullet:T^*M\times \mathcal V\to\mathcal V\quad\text{and}\quad 
\bullet:TM\times \mathcal V\to\mathcal V
\end{align*}
obtained from applying the functor $\mathcal G_0\times_{\mathcal G_0}\,\cdot\,$ to the natural Lie algebra representations $\mathfrak g_\pm\times\mathbb V\to\mathbb V$, respectively. Indeed, recall from, e.g., \cite{Book}*{Proposition 5.1.10 (2)} that there holds
\begin{align}\label{tractor connection priliminary}
    \nabla^{\mathcal V}_{\xi} s
    &=\nabla_\xi s +\xi\bullet s+\mathrm{P}(\xi)\bullet s
\end{align}
for every $s\in\Gamma(\mathcal V)$ and $\xi\in\mathfrak X(M)$.

Following the convention of Penrose abstract index notation, we use unprimed upper indices for $E$ and lower indices for $E^*$ and primed upper indices for $F$ and lower indices for $F^*$. Thus
\begin{align*}
   W=W{}^A_{A'}{}^B_{B'}{}^{C}_{D},\quad W'=W'{}^A_{A'}{}^B_{B'}{}^{C'}_{D'}\quad\text{and}\quad\tau=\tau{}^A_{A'}{}^B_{B'}{}^{C'}_D.
\end{align*}
A contraction is denoted by the writing the same index variable on the two slots being contracted. We specify the following contractions.
\begin{equation}\label{Weyl tensor Ricci type contraction}
   \begin{aligned}
    tr(W){}^A_{A'}{}^B_{B'}&=W{}^A_{A'}{}^I_{B'}{}^{B}_{I},\\
    tr(W'){}^A_{A'}{}^B_{B'}&=W'{}^A_{A'}{}^B_{I'}{}^{I'}_{B'}\\
    \text{and }tr(\iota_\tau\tau){}^A_{A'}{}^B_{B'}&=\tau{}^I_{I'}{}^A_{A'}{}^{J'}_J\,\tau{}^J_{J'}{}^B_{B'}{}^{I'}_I.
\end{aligned}
\end{equation}
Recall the following Bianchi-identity for a Cartan connection.
\begin{lemma}\label[lemma]{Bianchi for Cartan}
(See, e.g., \cite{Book}*{Proposition 1.5.9 (1)}) Let $(\mathcal G\to M,\omega)$ be a Cartan geometry of type $(G,P)$. Denote by $\mathcal AM=\mathcal G\times_P\mathfrak g$ its adjoint tractor, which has tractor connection $\nabla^{\mathcal A}$. The curvature $\kappa\in\Omega^2(M,\mathcal AM)$ of the Cartan connection $\omega$ satisfies
    \[\sum\Big(
    \nabla^{\mathcal A}_{\xi_1}(\kappa(\xi_2,\xi_3))-\kappa([\xi_1,\xi_2],\xi_3)
    \Big)=0\]
for all $\xi_i\in\mathfrak X(M),i=1,2,3$. Here, the sum runs over cyclic permutations of $\xi_1,\xi_2$ and $\xi_3$.
\end{lemma}
Denote by $[\ \ ]$ and $(\ \ )$ the alternation and the symmetrization of indices, respectively. For a symmetric $(0,2)$-tensor $S{}^A_{A'}{}^B_{B'}\in\Gamma(\mathrm{Sym}^2\,T^*M)$, we have that 
\[S{}^{(A}_{[A'}{}^{B)}_{B']}=S{}^{[A}_{(A'}{}^{B]}_{B')}=0.\]
Moreover, the decomposition
\[\mathrm{Sym}^2\,T^*M=\mathrm{Sym}^2\, E\otimes\mathrm{Sym}^2\,F^*\oplus\wedge^2\, E\otimes\wedge^2\,F^*\]
is given by
\[S{}^A_{A'}{}^B_{B'}=S{}^{(A}_{(A'}{}^{B)}_{B')}+S{}^{[A}_{[A'}{}^{B]}_{B']}.\]
Observe that $tr(\iota_\tau\tau){}^A_{A'}{}^B_{B'}=\tau{}^I_{I'}{}^A_{A'}{}^{J'}_J\,\tau{}^J_{J'}{}^B_{B'}{}^{I'}_I$ is a symmetric $(0,2)$-tensor.

We recap the relations between the three tensors in \eqref{Weyl tensor Ricci type contraction}.
It is originally stated in \cite{HSSS12B}*{Equation 26} without proof. We provide a proof for completeness.
\begin{proposition}\label{Weyl tensor}
(See \cite{HSSS12B}*{Equation 26})
Notation as in \eqref{Weyl tensor Ricci type contraction}. Then 
\[tr(W)=tr(W')\in\Gamma(\mathrm{Sym}^2\,T^*M).\]
This symmetric tensor is determined by
\begin{align*}
            tr(\iota_\tau\tau){}^{(A}_{(A'}{}^{B)}_{B')}=n\,tr(W){}^{(A}_{(A'}{}^{B)}_{B')}
            \quad\text{and}\quad
            tr(\iota_\tau\tau){}^{[A}_{[A'}{}^{B]}_{B']}=(n+4)\,tr(W){}^{[A}_{[A'}{}^{B]}_{B']}.
\end{align*}
\end{proposition}

\begin{proof}
Let $(\mathcal G\to M,\omega)$ be the normal parabolic geometry of type $(SL(n+2,\mathbb R),P)$ canonically associated to the AG-structure $(M,E,F)$. Denote its curvature by $\kappa$ and denote the chosen Weyl structure by $\sigma:\mathcal G_0\to\mathcal G$. Then
\begin{align}\label{decompose kappa}
\sigma^*\kappa=\left(\begin{array}{cc|cc}
	W{}^A_{A'}{}^B_{B'}{}^{C}_{D}&&Y{}^A_{A'}{}^B_{B'}{}^{C}_{D'}&\\
 \hline&&&\\
	\tau{}^A_{A'}{}^B_{B'}{}^{C'}_{D}&&W'{}^A_{A'}{}^B_{B'}{}^{C'}_{D'}&\\
 &&&
\end{array}\right)
\end{align}
where $Y$ is the Cotton-York tensor of the chosen Weyl structure; see \eqref{induced Weyl objects}. Using \eqref{codiff}, one computes that
\begin{align}\label{identities on Weyl objects}
\frac 1 2\partial^*_\mathfrak g\sigma^*\kappa=\left(\begin{array}{cc|cc}
	\tau{}^A_{A'}{}^B_{I'}{}^{I'}_{C}&&W'{}^A_{A'}{}^B_{I'}{}^{I'}_{C'}-W{}^A_{A'}{}^I_{C'}{}^{B}_{I}&\\
 \hline&&&\\
	0&&-\tau{}^A_{A'}{}^I_{C'}{}^{B'}_{I}&\\
 &&&
\end{array}\right)
\end{align}
which is zero by the assumption of normality. In particular, we see that $tr(W)=tr(W')$.
   
Applying the functor $\mathcal G_0\times_{G_0}\,\cdot\,$ to the grading decomposition \eqref{gradingAG} for $\mathfrak{sl}(n+2,\mathbb R)$ yields the decomposition
\[\mathcal AM=TM\oplus \mathfrak s(L(E,E)\oplus L(F,F))\oplus T^*M\] 
associated to the Weyl structure $\sigma$; see \eqref{frak s}.
We use \eqref{tractor connection priliminary} to compute the tractor connection on the adjoint tractor bundle. Then the Bianchi identity (cited in \cref{Bianchi for Cartan}) in the $TM$-component reads
\begin{align*}
\sum\Big(\nabla_\xi(\tau(\eta,\zeta))+\xi\bullet((W,W')(\eta,\zeta))+\tau([\xi,\eta],\zeta)\Big)=0.
\end{align*}
Here, the sum runs over cyclic permutations of $\xi,\eta,\zeta\in\mathfrak X(M)$. Since $\tau([\xi,\eta],\zeta)=\tau(\nabla_\xi\eta-\nabla_\eta\xi-\tau(\xi,\eta),\zeta)$ and $(\nabla_\xi\tau)(\eta,\zeta)=\nabla_\xi(\tau(\eta,\zeta))-\tau(\nabla_\xi\eta,\zeta)-\tau(\eta,\nabla_\xi\zeta)$, we obtain
\[\sum\Big((\nabla_\xi\tau)(\eta,\zeta)+\tau(\tau(\xi,\eta),\zeta)+\xi\bullet((W,W')(\eta,\zeta))\Big)=0.\]
Using indices, this reads
\[\sum\Big(\nabla^A_{A'}\tau{}^B_{B'}{}^C_{C'}{}^{D'}_{D}
+\tau{}^A_{A'}{}^B_{B'}{}^{I'}_{I}\,\tau{}^I_{I'}{}^C_{C'}{}^{D'}_{D}
+\delta^{D'}_{A'}\,W{}^B_{B'}{}^C_{C'}{}^{A}_{D}
-\delta^{A}_{D}\,W'{}^B_{B'}{}^C_{C'}{}^{D'}_{A'}
\Big)=0 \]
where the sum runs over cyclic permutations of pairs $(A,A'),(B,B'),(C,C')$. Now we contract indices $D$ with $C$ and $D'$ with $B'$, rewrite the index $C'$ as $B'$, and use the identities in \eqref{identities on Weyl objects}. The resulting expression is
\begin{align*}
\tau{}^B_{J'}{}^J_{B'}{}^{I'}_{I}\,\tau{}^I_{I'}{}^A_{A'}{}^{J'}_{J}
+tr(W){}^B_{A'}{}^A_{B'}+tr(W){}^A_{B'}{}^B_{A'}-(n+2)\,tr(W){}^A_{A'}{}^B_{B'}=0.
\end{align*}
Recall from \eqref{harmonic components AG} that $\tau{}^A_{A'}{}^B_{B'}{}^{C'}_{D}=\tau{}^B_{A'}{}^A_{B'}{}^{C'}_{D}$ on an AG-structure of type $(2,n)$, $n>2$. Hence 
\[\tau{}^B_{J'}{}^J_{B'}{}^{I'}_{I}\,\tau{}^I_{I'}{}^A_{A'}{}^{J'}_{J}=tr(\iota_\tau\tau){}^A_{A'}{}^B_{B'}.\] Symmetrizing and alternating the indices, we obtain
\begin{align*}
            tr(\iota_\tau\tau){}^{(A}_{(A'}{}^{B)}_{B')}&=n\,tr(W){}^{(A}_{(A'}{}^{B)}_{B')}&
            tr(\iota_\tau\tau){}^{[A}_{[A'}{}^{B]}_{B']}&=(n+4)\,tr(W){}^{[A}_{[A'}{}^{B]}_{B']}\\
            tr(\iota_\tau\tau){}^{(A}_{[A'}{}^{B)}_{B']}&=tr(W){}^{(A}_{[A'}{}^{B)}_{B']}=0&
            tr(\iota_\tau\tau){}^{[A}_{(A'}{}^{B]}_{B')}&=tr(W){}^{[A}_{(A'}{}^{B]}_{B')}=0.
\end{align*}
This proves the statement of the proposition.
\end{proof}

\section{Tractor calculus for AG-structures}\label{Tractor calculus for almost Grassmann structures}
Let $(M,E,F)$ be an AG-structure of type $(2,n)$ and $(\mathcal G\to M,\omega)$ be the normal parabolic geometry of type $(SL(n+2,\mathbb R),P)$ canonically associated to it. Denote by $\mathcal T=\mathcal G\times_P\mathbb R^{n+2}$ the standard tractor bundle of $(\mathcal G\to M,\omega)$. Its dual $\mathcal T^*$ is the standard cotractor bundle of $(\mathcal G\to M,\omega)$. Then there are natural short exact sequences
\begin{align*}
    0\to E\to\mathcal T\to F\to 0\quad \text{and}\quad 0\to F^*\to\mathcal T^*\to E^*\to 0.
\end{align*}

Recall that the space of Weyl structures of $(\mathcal G\to M,\omega)$ is affine over $\Omega^1(M)$. Indeed, fix a Weyl structure $\sigma:\mathcal G_0\to\mathcal G$. Then $\hat\sigma:\mathcal G_0\to\mathcal G$ is a Weyl structure if and only if there exists
\[\Upsilon\in C^\infty(\mathcal G_0,\mathfrak p_+)^{G_0}=\Omega^1(M)\]
such that $\hat\sigma(u)=\sigma(u).\exp{\Upsilon(u)}$ for all $u\in\mathcal G_0$; see, e.g., \cite{Book}*{Proposition 5.1.1}.

Applying the functor $\mathcal G_0\times_{G_0}\,\cdot\,$ to the obvious $G_0$-module decomposition $\mathbb R^{n+2}=\mathbb R^n\oplus\mathbb R^2$ and its dual yields the splitting decompositions
\begin{align*}
\mathcal T\stackrel{\sigma}{\cong}
F\oplus E
\quad\text{and}\quad
\mathcal T^*\stackrel{\sigma}{\cong}
E^*\oplus F^*
\end{align*}
associated to the Weyl structure $\sigma$. Under such identifications, each section of $\mathcal T$ can be uniquely written as 
\[(\eta,\xi)=(\eta,\xi)_\sigma\in\Gamma(F\oplus E)\] and each section of $\mathcal T^*$ can be uniquely written as 
\[(\varphi,\mu)=(\varphi,\mu)^\sigma\in\Gamma(E^*\oplus F^*).\] Let $\hat\sigma$ be a Weyl structure determined by $\sigma$ and $\Upsilon{}^A_{A'}\in\Omega^1(M)$ as above. The identifications changes by
\begin{enumerate}
    \item [$\bullet$]$(\eta^{A'},\xi^A)_\sigma=
    (\eta^{A'},\xi^A-\Upsilon{}^A_{I'}\,\eta^{I'})_{\hat\sigma}\in\Gamma(\mathcal T)$ and
    \item [$\bullet$]$(\varphi_A,\mu_{A'})^{\sigma}=
(\varphi_A,\mu_{A'}+\Upsilon{}^I_{A'}\,\varphi_I)^{\hat\sigma}\in\Gamma(\mathcal T^*)$.
\end{enumerate}
Let $\nabla$ and $\hat\nabla$ be the Weyl connections of the Weyl structures $\sigma$ and $\hat\sigma$, respectively. Using \cite{Book}*{Proposition 5.1.6} one computes that
\begin{enumerate}
\item [$\bullet$] $\hat\nabla{}^A_{A'}\eta^{B'}=\nabla{}^A_{A'}\eta^{B'}+\Upsilon{}^A_{I'}\,\eta^{I'}\,\delta{}^{B'}_{A'}$ for $\eta\in\Gamma(F)$ and
\item [$\bullet$] 
$\hat\nabla{}^A_{A'}\varphi_B=\nabla{}^A_{A'}\varphi_B+\Upsilon{}^I_{A'}\,\varphi_I\,\delta{}^A_B$ for $\varphi\in\Gamma(E^*)$.
\end{enumerate}
By computation using formula \eqref{tractor connection priliminary}, the tractor connection $\nabla^{\mathcal T}$ on $\mathcal T$ and the tractor connection $\nabla^{\mathcal T^*}$ on $\mathcal T^*$ are given by

\begin{enumerate}
\item [$\bullet$] $ \nabla^{\mathcal T}{}^A_{A'}(\eta,\xi)_\sigma^{B'B}=(\nabla^{A}_{A'}\eta^{B'}+\xi^A\,\delta^{B'}_{A'},\nabla^{A}_{A'}\xi^{B}+\eta^{I'}\,\mathrm P{}^{A}_{A'}{}^{B}_{I'})_\sigma$ and
\item [$\bullet$]$\nabla^{\mathcal T^*}{}^A_{A'}(\varphi,\mu)^\sigma_{BB'}=(\nabla^{A}_{A'}\varphi_{B}-\mu_{A'}\,\delta^{A}_{B},\nabla^{A}_{A'}\mu_{B'}-\varphi_{I}\,\mathrm P{}^{A}_{A'}{}^{I}_{B'})^\sigma$.
\end{enumerate}
One may find all of the above formulae in, e.g., \cite{Book}*{Subsection 5.1.11}. Here, $\mathrm P\in\Omega^1(M,T^*M)$ is the Rho tensor induced by the Weyl structure $\sigma$. Let $\hat{\mathrm P}$ be the Rho tensor induced by $\hat\sigma$. Using \cite{Book}*{Proposition 5.1.8} one computes that
\begin{enumerate}
    \item [$\bullet$]$\hat{\mathrm P}{}^{A}_{A'}{}^{B}_{B'}=\mathrm P{}^{A}_{A'}{}^{B}_{B'}+\nabla{}^A_{A'}\Upsilon{}^{B}_{B'}-\Upsilon{}^B_{A'}\,\Upsilon{}^A_{B'}$
\end{enumerate}
where $\nabla$ is the Weyl connection induced by $\sigma$.

The splitting operator $L$ as well as the first BGG operator $D$ of any tractor bundle arise naturally from deriving necessary conditions for a section being parallel. Originally, these operators have been introduced in \cite{CSS}*{Theorem 2.5}. We provide explicit formulae of them for the cases of the standard tractor $\mathcal T$ and the standard cotractor $\mathcal T^*$. The natural projection $\Pi:\mathcal T\twoheadrightarrow F$ induces
\[id_{T^*M}\otimes\Pi:T^*M\otimes\mathcal T\twoheadrightarrow T^*M\otimes F.\]
There is a natural decomposition
\begin{align*}
    T^*M\otimes F\xrightarrow{\cong}& E\otimes\mathfrak{sl}(F)\oplus E\\
    \Phi\mapsto&(\Phi_o,\frac 1 n\, tr(\Phi)),
\end{align*}
where
\begin{align*}
    (\Phi_o){}^A_{A'}{}^{B'}=\Phi{}^A_{A'}{}^{B'}-\frac 1 n\,\delta{}^{B'}_{A'}\,\Phi{}^A_{I'}{}^{I'}\quad\text{and}\quad tr(\Phi)^A=\Phi{}^A_{I'}{}^{I'}.
\end{align*}
Similarly, the natural projection $\Pi:\mathcal T^*\twoheadrightarrow E^*$ induces
\[id_{T^*M}\otimes\Pi:T^*M\otimes\mathcal T^*\twoheadrightarrow T^*M\otimes E^*.\]
There is a natural decomposition
\begin{align*}
    T^*M\otimes E^*\xrightarrow{\cong}& F^*\otimes\mathfrak{sl}(E)\oplus F^*\\
    \Psi\mapsto
    &(\Psi_o,\frac 1 2\,tr(\Psi)),
\end{align*}
where
\begin{align*}
    (\Psi_o){}^A_{A'}{}_{B}=\Psi{}^A_{A'}{}_{B}-\frac 1 2\,\delta{}^{A}_{B}\,\Psi{}^I_{A'}{}_{I}\quad\text{and}\quad 
    tr(\Psi)_{A'}=\Psi{}^I_{A'}{}_{I}.
\end{align*}

\begin{proposition}\label{splitting operator}
Notation as above. Then the splitting operator and the first BGG operator of the standard tractor bundle $\mathcal T$ are given by
\begin{align*}
    L:\Gamma(F)\to\Gamma(\mathcal T),\quad    \eta\mapsto(\eta,-\frac 1 n\,tr(\nabla\eta))_\sigma
\end{align*}
and
\begin{align*}
    D:\Gamma(F)\to\Gamma(E\otimes\mathfrak {sl}(F)),\quad\,\eta\mapsto(\nabla\eta)_o,
\end{align*}
respectively. Similarly, the splitting operator and the first BGG operator of the standard cotractor bundle $\mathcal T^*$ are given by
\begin{align*}
    L:\Gamma(E^*)\to\Gamma(\mathcal T^*),\quad    \varphi\mapsto(\varphi,\frac 1 2\,tr(\nabla\varphi))_\sigma
\end{align*}
and
\begin{align*}
    D:\Gamma(E^*)\to\Gamma( F^*\otimes\mathfrak {sl}(E)),\quad\,\varphi\mapsto(\nabla\varphi)_o,
\end{align*}
respectively. 

Here, $\nabla$ is the Weyl connection of the Weyl structure $\sigma$. The expressions are independent of the choice of Weyl structure.
\end{proposition}
\begin{proof}
Notation as in the paragraphs preceding the statement of the proposition. Note that the restriction of $id_{T^*M}\otimes\Pi$ to $(id_{T^*M}\otimes\Pi)^{-1}(E\otimes\mathfrak {sl}(F))$ is precisely the natural projection $\ker\,\partial^*\twoheadrightarrow\ker\,\partial^*/\mathrm{im}\,\partial^*$ in the language of \cite{CSS}. In particular, the splitting operator $L:\Gamma(F)\to\Gamma(\mathcal T)$ of $\mathcal T$ is uniquely determined by
\begin{enumerate}
    \item [$\bullet$]$\Pi(L\,\eta)=\eta$ and
    \item [$\bullet$]$(id_{T^*M}\otimes\Pi)(\nabla^{\mathcal T}(L\,\eta))\in\Gamma(E\otimes\mathfrak {sl}(F))$
\end{enumerate}
for all $\eta\in\Gamma(F)$; see \cite{CSS}*{Lemma 2.7 (2)}. Write $L\,\eta=(\eta,\xi)_\sigma$, then 
\begin{align*}
    (id_{T^*M}\otimes\Pi)(\nabla^{\mathcal T}(L\,\eta))=\nabla{}^{A}_{A'}\eta^{B'}+\xi^A\,\delta^{B'}_{A'}
\end{align*}
has trace $0=\nabla{}^{A}_{I'}\eta^{I'}+n\,\xi^A$. Hence $L\,\eta=(\eta^{A'},-\frac 1 n\,\nabla^A_{I'}\eta^{I'})_\sigma$.

The first BGG operator $D:\Gamma(F)\to\Gamma(E\otimes\mathfrak {sl}(F))$ of $\mathcal T$ is defined by
\begin{enumerate}
    \item [$\bullet$]$D\,\eta=(id_{T^*M}\otimes\Pi)(\nabla^{\mathcal T}(L\,\eta))$
\end{enumerate}
for all $\eta\in\Gamma(F)$; see \cite{CSS}*{Theorem 2.5}. It follows from the formulae of $\nabla^{\mathcal T}$ and of $L$ that $D\,\eta=(\nabla\eta)_o$.

Following analogous steps, one obtains the expressions and properties of the splitting operator $L:\Gamma(E^*)\to\Gamma(\mathcal T^*)$ and the first BGG operator $D:\Gamma(E^*)\to\Gamma(F^*\otimes\mathfrak {sl}(E))$ of $\mathcal T^*$ as in the statement. Recall that $L$ is uniquely determined by
\begin{enumerate}
     \item [$\bullet$]$\Pi(L\,\varphi)=\varphi$ and
    \item [$\bullet$]$(id_{T^*M}\otimes\Pi)(\nabla^{\mathcal T^*}(L\,\varphi))\in\Gamma(F^*\otimes\mathfrak {sl}(E))$
\end{enumerate}
for all $\varphi\in\Gamma(E^*)$. $D$ is defined by
\begin{enumerate}
    \item [$\bullet$] $D\,\varphi=(id_{T^*M}\otimes\Pi)(\nabla^{\mathcal T^*}(L\,\varphi))$
\end{enumerate}
for all $\varphi\in\Gamma(E^*)$.
\end{proof}
\begin{remark}[Solutions and normal solutions of BGG]\label[remark]{normal solutions}
Let $\mu\in\Gamma(\mathcal T)$ and $\nu\in\Gamma(\mathcal T^*)$ project to $\eta=\Pi(\mu)\in\Gamma(F)$ and $\varphi=\Pi(\nu)\in\Gamma(E^*)$, respectively. By the defining properties of $L$ and $D$ given in the proof of \cref{splitting operator}, one has that
\begin{align*}
\nabla^{\mathcal T}\mu=0\quad\Longrightarrow\quad    \nabla^{\mathcal T}\mu\in\Omega^1(M,E)\subseteq\Omega^1(M,\mathcal T)\quad&
\iff\quad (\mu=L\,\eta\text{ and }D\,\eta=0)\end{align*}
and
\begin{align*}
\nabla^{\mathcal T^*}\nu=0\quad\Longrightarrow\quad    \nabla^{\mathcal T^*}\nu\in\Omega^1(M,F^*)\subseteq\Omega^1(M,\mathcal T^*)\quad&
\iff\quad (\nu=L\,\varphi\text{ and }D\,\varphi=0).
\end{align*}
We say that $\eta$ is a solution of $D$ if $D\,\eta=0$. If there holds $\nabla^{\mathcal T}\mu=0$ in addition, we say that $\eta$ is a normal solution of $D$. Analogous terminologies apply to $\varphi$.
\end{remark}
We prove the following surjectivity property that holds for any solution of one of the two particular BGG operators. It is a consequence of certain surjectivity conditions on the underlying $G$-representations and thus need not hold for solutions of first BGG operators in general. For normal solutions, the analogous result can be deduced more easily by analyzing the situation of the homogeneous model; see \cref{curved orbit} below.
\begin{corollary}\label{surjectivity}
Let $\nabla$ be any Weyl connection of $(M,E,F)$. Assume that $\eta\in\Gamma(F)$, respectively, $\varphi\in\Gamma(E^*)$ is a solution of the corresponding first BGG operator. Then for $x\in M$ such that $\eta(x)=0$ and $L\,\eta(x)\neq 0$, the linear map
\[\nabla\eta(x):T_xM\to F_x\]
is surjective and independent of the choice of Weyl connection.

Similarly, for $x\in M$ such that $\varphi(x)=0$ and $L\,\varphi(x)\neq 0$, the linear map
\[\varphi(x):T_xM\to E^*_x\]
is surjective and independent of the choice of Weyl connection.
\end{corollary}

\begin{proof}
That $D\,\eta=(\nabla\eta)_o=0$ implies that $\nabla\eta\in\Gamma(E\otimes F^*\otimes F)$ is a pure trace given by
\[\nabla{}^A_{A'}\eta^{B'}=\frac 1 n\delta{}^{B'}_{A'}\nabla{}^A_{I'}
\eta^{I'}.\]
In particular, for $x\in M$ such that $\eta(x)=0$, there holds $(L\,\eta)(x)\in E_x$ and
\begin{align*}
\nabla\eta(x)=-(L\,\eta)(x)\otimes id_{F_x}\in E_x\otimes F^*_x\otimes F_x.
\end{align*}
This proves the statement for $\eta$.

The statement for $\varphi$ follows in analogous steps. Note that for $x\in M$ such that $\varphi(x)=0$, there holds $(L\,\varphi)(x)\in F^*_x$ and
\begin{align*}
\nabla\varphi(x)=(L\,\varphi)(x)\otimes id_{E^*_x}\in F^*_x\otimes E_x\otimes E^*_x.
\end{align*}
\end{proof}

By a direct computation in the spirit of \cite{EM}*{Section 3}, we obtain the following relations.

\begin{lemma}\label[lemma]{Phi}
Notation as in \eqref{Weyl tensor Ricci type contraction}. Let $\Phi,\Psi\in\Gamma(\mathrm{Sym}^2\,T^*M)$ be defined by
\begin{align*}
\Phi{}^{A}_{A'}{}^{B}_{B'}&=-\frac{1}{n-1} tr(W){}^{(A}_{(A'}{}^{B)}_{B')}-\frac{1}{n+1} tr(W){}^{[A}_{[A'}{}^{B]}_{B']}\\
&=-\frac{1}{(n-1)\,n} tr(\iota_\tau\tau){}^{(A}_{(A'}{}^{B)}_{B')}-\frac{1}{(n+1)\,(n+4)} tr(\iota_\tau\tau){}^{[A}_{[A'}{}^{B]}_{B']}\quad\text{and}\\
\Psi{}^{A}_{A'}{}^{B}_{B'}&=-tr(W){}^{(A}_{(A'}{}^{B)}_{B')}
-\frac{1}{3} 
tr(W){}^{[A}_{[A'}{}^{B]}_{B']}\\
&=-\frac{1}{n}\,tr(\iota_\tau\tau){}^{(A}_{(A'}{}^{B)}_{B')}-\frac{1}{3\,(n+4)}tr(\iota_\tau\tau){}^{[A}_{[A'}{}^{B]}_{B']},
\end{align*}
respectively; see \cref{Weyl tensor}. The following hold;  see \cref{normal solutions}.
\begin{enumerate}
    \item [(i)] Assume that $D\,\eta=0$. Then $\nabla^{\mathcal T}{}^A_{A'}(L\,\eta)^{B}=\Phi{}^{A}_{A'}{}^{B}_{I'}\eta^{I'}\in\Omega^1(M,E)\subseteq\Omega^1(M,\mathcal T)$.
    \item [(ii)] Assume that $D\,\varphi=0$. Then $\nabla^{\mathcal T^*}{}^A_{A'}(L\,\varphi)_{B'}=\Psi{}^{A}_{A'}{}^{I}_{B'}\varphi_I\in\Omega^1(M,F^*)\subseteq\Omega^1(M,\mathcal T^*)$.
\end{enumerate}
\end{lemma}
\begin{proof}
Let $\eta\in\Gamma(F)$ such that $D\,\eta=0$. Define $\Phi(\eta){}^{A}_{A'}{}^{B}\in\Omega^1(M,E)$ as the unique one-form such that
\begin{align*}
\nabla^{\mathcal T}{}^B_{B'}(L\,\eta)^{C'C}=(0^{C'},\Phi(\eta){}^{B}_{B'}{}^{C}).
\end{align*}
Let $\xi{}^{B'}_B\in\mathfrak X(M)$. Applying $\nabla^{\mathcal T}$ to $(0,\Phi(\eta)(\xi))^{C'C}=(0^{C'},\xi{}^{B'}_B
\Phi(\eta){}^{B}_{B'}{}^{C})\in\Gamma(\mathcal T)$, we obtain
\begin{align*}
\nabla^{\mathcal T}{}^A_{A'}(0,\Phi(\eta)(\xi))^{C'C}
&=\left(\xi{}^{B'}_B
\Phi(\eta){}^{B}_{B'}{}^{A}\delta{}^{C'}_{A'},
(\nabla{}^A_{A'}\xi{}^{B'}_B)\Phi(\eta){}^{B}_{B'}{}^{C}+\xi{}^{B'}_B(\nabla{}^A_{A'}
\Phi(\eta){}^{B}_{B'}{}^{C})\right)
\end{align*}
thus
\begin{align*}
R(\nabla^{\mathcal T}){}^{A}_{A'}{}^{B}_{B'}(L\,\eta)^{C'C}&=(\Phi(\eta){}^{B}_{B'}{}^{A}\delta{}^{C'}_{A'}-\Phi(\eta){}^{A}_{A'}{}^{B}\delta{}^{C'}_{B'},\\
&\quad\tau{}^A_{A'}{}^{B}_{B'}{}^{I'}_{I}\Phi(\eta){}^{I}_{I'}{}^{C}+
\nabla{}^A_{A'}
\Phi(\eta){}^{B}_{B'}{}^{C}
-\nabla{}^B_{B'}
\Phi(\eta){}^{A}_{A'}{}^{C})
\end{align*}
where $R(\nabla^{\mathcal T})$ is the curvature of $\nabla^{\mathcal T}$. On the other hand, applying the Cartan curvature $\kappa$ to
\[(L\,\eta)^{C'C}=(\eta^{C'},-\frac 1 n\nabla{}^C_{I'}\eta^{I'}),\]
we obtain
\begin{align*}
\kappa{}^{A}_{A'}{}^{B}_{B'}(L\,\eta)^{C'C}&=(-\frac 1 n\tau{}^{A}_{A'}{}^{B}_{B'}{}^{C'}_I\nabla{}^I_{I'}\eta^{I'}+W'{}^{A}_{A'}{}^{B}_{B'}{}^{C'}_{I'}\eta^{I'},\\
&\quad-\frac 1 nW{}^{A}_{A'}{}^{B}_{B'}{}^{C}_I\nabla{}^I_{I'}\eta^{I'}+Y{}^{A}_{A'}{}^{B}_{B'}{}^{C}_{I'}\eta^{I'});
\end{align*}
see \eqref{decompose kappa}. Recall from, e.g., \cite{Book}*{Corollary 1.5.7} that $\kappa=R(\nabla^{\mathcal T})$. In particular, there holds
\begin{align*}
\Phi(\eta){}^{B}_{B'}{}^{A}\delta{}^{C'}_{A'}-\Phi(\eta){}^{A}_{A'}{}^{B}\delta{}^{C'}_{B'}=-\frac 1 n\tau^{A}_{A'}{}^{B}_{B'}{}^{C'}_I\nabla{}^I_{I'}\eta^{I'}+W'{}^{A}_{A'}{}^{B}_{B'}{}^{C'}_{I'}\eta^{I'}. 
\end{align*}
Contracting indices $B'$ with $C'$ results in
\begin{align*}
\Phi(\eta){}^{B}_{A'}{}^{A}-n\,\Phi(\eta){}^{A}_{A'}{}^{B}=W'{}^{A}_{A'}{}^{B}_{J'}{}^{J'}_{I'}\eta^{I'}=tr(W){}^{A}_{A'}{}^{B}_{I'}\eta^{I'};
\end{align*}
see \eqref{Weyl tensor Ricci type contraction}. We symmetrize and alternate indices $A$ with $B$ and sum up these two expressions. The resulting expression is
\begin{align*}
\Phi(\eta){}^A_{A'}{}^{B}=-\frac{1}{n-1} 
tr(W){}^(
{}^{A}_{A'}{}^{B}_{I'}{}^)\eta^{I'}
-\frac{1}{n+1} 
tr(W){}^[
{}^{A}_{A'}{}^{B}_{I'}{}^]
\eta^{I'}=\Phi{}^A_{A'}{}^{B}_{I'}\,\eta^{I'}
\end{align*}
where $\Phi{}^A_{A'}{}^{B}_{B'}$ is defined in the statement of the lemma. We have used the identities in \cref{Weyl tensor}. This completes the proof of (i).

Similarly, for every $\varphi\in\Gamma(E^*)$ such that $D\,\varphi=0$, we define $\Psi(\varphi){}^{A}_{A'}{}_{B'}\in\Omega^1(M,F^*)$ to be the one-form such that $\nabla^{\mathcal T^*}{}^B_{B'}(L\,\varphi)_{CC'}=(0_{C},\Psi(\varphi){}^{B}_{B'}{}_{C'})$. Recall that $(L\,\varphi)_{CC'}=(\varphi_C,\frac 1 2\,\nabla{}^I_{C'}\varphi_I)$. Following analogous steps, we obtain
\begin{align*}
R(\nabla^{\mathcal T}){}^{A}_{A'}{}^{B}_{B'}(L\,\varphi)_{CC'}&=(-\Psi(\varphi){}^B_{B'}{}_{A'}\,\delta{}^A_C+\Psi(\varphi){}^A_{A'}{}_{B'}\,\delta{}^B_C,\\
&\quad
\tau{}^A_{A'}{}^{B}_{B'}{}^{I'}_{I}\Psi(\varphi){}^{I}_{I'}{}_{C'}+
\nabla{}^A_{A'}
\Psi(\varphi){}^{B}_{B'}{}_{C'}
-\nabla{}^B_{B'}
\Psi(\varphi){}^{A}_{A'}{}_{C'})\quad\text{and}\\
\kappa{}^{A}_{A'}{}^{B}_{B'}(L\,\varphi)_{CC'}&=(-\varphi_I\,W{}^{A}_{A'}{}^{B}_{B'}{}^I_C-\frac 1 2\tau{}^{A}_{A'}{}^{B}_{B'}{}^{I'}_C\,\nabla{}^I_{I'}\,\varphi_I,\\
&\quad-\varphi_I\,Y{}^{A}_{A'}{}^{B}_{B'}{}^I_{C'}-\frac 1 2W'{}^{A}_{A'}{}^{B}_{B'}{}^{I'}_{C'}\,\nabla{}^I_{I'}\,\varphi_I
).
\end{align*}
Contracting indices $B$ with $C$ of the first slot of these two equaling expressions results in
\[\Psi(\varphi){}^A_{A'}{}_{B'}= 
-tr(W){}_(
{}^{A}_{A'}{}^{B}_{I'}{}_)\varphi_I
-\frac{1}{3} 
tr(W){}_[
{}^{A}_{A'}{}^{B}_{I'}{}_]\varphi_I=\Psi{}^{A}_{A'}{}^{I}_{B'}\varphi_I.\]
\end{proof}
One direct consequence of \cref{Phi} is a characterization of all parallel sections of the standard tractor bundle $\mathcal T$ and the standard cotractor bundle $T^*$.
\begin{theorem}[Normal solutions of the first BGG operators] \label[theorem]{parallel tractor}
Notation as in \cref{splitting operator}.
\begin{enumerate}
    \item [(i)]For every $\eta\in\Gamma(F)$, there holds $\nabla^{\mathcal T}(L\,\eta)=0$ if and only if $D\,\eta=0$ and $tr(\iota_\tau\tau){}
        ^{A}_{A'}{}^{B}_{B'}\eta^{B'}=0$.
    \item [(ii)]For every $\varphi\in\Gamma(E^*)$, there holds $\nabla^{\mathcal T^*}(L\,\varphi)=0$ if and only if $D\,\varphi=0$ and $tr(\iota_\tau\tau){}
        ^{A}_{A'}{}^{B}_{B'}\varphi_B=0$.
\end{enumerate}
    
\end{theorem}
\begin{proof}
    This is a direct consequence of \cref{Phi}. Note that
    \begin{align*}
        \Phi{}^{A}_{A'}{}^{B}_{I'}\eta^{I'}=0\quad\iff\quad
        tr(\iota_\tau\tau){}^{A}_{A'}{}^{B}_{I'}\eta^{I'}=0
    \end{align*} 
    and the analogous property for $\Psi$.
\end{proof}

Another direct consequence of \cref{Phi} is a one-to-one correspondence between solutions of the first BGG operator and parallel sections of a certain prolongation connection on the corresponding tractor bundle; see \cref{prolongation connection} below. These connections on $\mathcal T$ and $\mathcal T^*$ have been obtained using a different method in \cite{HSSS12B}; see \cref{general algorithm} below. We also refer the reader to, e.g., \cite{EM}*{Theorem 3.1}, \cite{CEMN20}*{Propositions 3.2 - 3.3} and \cite{HSSS12B}*{Sections 3, 4 and 7} for prolongation connections on different geometric structures and to \cite{HSSS12A} and \cite{HSSS12B} for a general theory of prolongation connections.

We also want to show that our prolongation connections are tractor connections of some Cartan connections. 
Let $(M,E,F)$ be an AG-structure of type $(2,n)$ and $(\mathcal G\to M,\omega)$ the normal parabolic geometry canonically associated to it. It follows from, e.g., \cite{Book}*{Proposition 3.1.10 (1)} that for $\hat\omega\in\Omega^1(\mathcal G,\mathfrak g)$ such that
\begin{align}\label{change Cartan connection}
    \hat\omega-\omega\in\Omega^1_{hor}(\mathcal G,\mathfrak g_1)^P\cong\Omega^1(M,T^*M),
\end{align}
$(\mathcal G\to M,\hat\omega)$ is a parabolic geometry with the same underlying AG-structure $(M,E,F)$. Here, the isomorphism in \eqref{change Cartan connection} is determined by the projection of $\omega$ in $\Omega^1(\mathcal G,\mathfrak g/\mathfrak p)$ together with the isomorphism $(\mathfrak g/\mathfrak p)^*\cong\mathfrak g_1$ induced by the Killing form on $\mathfrak g$. Note that the isomorphism $\Omega^1_{hor}(\mathcal G,\mathfrak g_1)^P\cong\Omega^1(M,T^*M)$ determined by $\hat\omega$ agrees with that determined by $\omega$.

\begin{theorem}[Prolongation connections]\label[theorem]{prolongation connection}
Let $\nabla^{\mathcal T}$ and $\nabla^{\mathcal T^*}$ be tractor connections as in the beginning of the section and $\Phi,\Psi\in\Omega^1(M,T^*M)$ be as in \cref{Phi}.
\begin{enumerate}
    \item [(i)] For every $(\eta,\xi)\in\Gamma(F\oplus E)\cong\Gamma(\mathcal T)$,
\begin{align}\label{tractor prolongation formula}
    \hat\nabla^{\mathcal T}{}^A_{A'}(\eta,\xi)^{B'B}=\nabla^{\mathcal T}{}^A_{A'}(\eta,\xi)^{B'B}-(0^{B'},\Phi{}^{A}_{A'}{}^{B}_{I'}\eta^{I'})
\end{align}
defines a linear connection $\hat\nabla^{\mathcal T}$ on the standard tractor bundle $\mathcal T$ such that
\begin{align}\label{general theory}
\hat\nabla^{\mathcal T}\mu=0\quad\iff\quad (\mu=L\,\eta\text{ and }D\,\eta=0)
\end{align}
for every $\eta\in\Gamma(F)$ and $\mu=(\eta,*)\in\Gamma(\mathcal T)$.
Moreover, $\hat\nabla^{\mathcal T}$ is the tractor connection induced by the modified Cartan connection $\hat\omega$ where 
\[\hat\omega-\omega=-\Phi\in\Omega^1(M,T^*M);\]
see \eqref{change Cartan connection}.
\item [(ii)] For every $(\varphi,\mu)\in\Gamma(E^*\oplus F^*)\cong\Gamma(\mathcal T^*)$,
\begin{align}\label{cotractor prolongation formula}
    \tilde\nabla^{\mathcal T^*}{}^A_{A'}(\varphi,\mu)_{BB'}=\nabla^{\mathcal T^*}{}^A_{A'}(\varphi,\mu)_{BB'}-(0_B,\Psi{}^{A}_{A'}{}^{I}_{B'}\varphi_I)
\end{align}
defines a linear connection $\tilde\nabla^{\mathcal T^*}$ on the standard cotractor bundle $\mathcal T^*$ such that
\begin{align}\label{general theory cotractor}
\tilde\nabla^{\mathcal T^*}\nu=0\quad\iff\quad (\nu=L\,\varphi\text{ and }D\,\varphi=0)
\end{align}
for every $\varphi\in\Gamma(E^*)$ and $\nu=(\varphi,*)\in\Gamma(\mathcal T^*)$. Moreover, $\tilde\nabla^{\mathcal T^*}$ is the tractor connection induced by the modified Cartan connection $\tilde\omega$ where 
\[\tilde\omega-\omega=\Psi\in\Omega^1(M,T^*M);\]
see \eqref{change Cartan connection}.
\end{enumerate}
\end{theorem}
\begin{proof}
If $\eta\in\Gamma(F)$ such that $D\,\eta=0$, it follows from \cref{Phi} that $\hat\nabla^{\mathcal T}(L\,\eta)=0$. Conversely, let 
\[\mu=(\eta,\xi)\in\Gamma(F\oplus E)\cong\Gamma(\mathcal T).\] If $\hat\nabla^{\mathcal T}\mu=0$, then \[\nabla^{\mathcal T}\mu\in\Omega^1(M,E)\subseteq \Omega^1(M,\mathcal T).\]
This implies that $\mu=L\,\eta$ and $D\,\eta=0$; see \cref{normal solutions}. Moreover, that $\hat\nabla^{\mathcal T}$ is the tractor connection of $\hat\omega$ follows from a direct computation using \eqref{tractor connection general formula}. This proves (i). Then (ii) is proved in an analogous way.
\end{proof}
\begin{remark}\label[remark]{general algorithm}
The authors of \cite{HSSS12A} have introduced a normality condition on prolongation connections associated to first BGG operators and proved that there is a unique such connection on a tractor bundle. Precisely, let $(\mathcal V,\nabla^{\mathcal V})$ be any tractor bundle on a manifold $M$ with associated normal tractor connection. Since the space of linear connections $\hat\nabla^{\mathcal V}$ on $\mathcal V$ is an affine space over $T^*M\otimes L(\mathcal V,\mathcal V)$, $\hat\nabla^{\mathcal V}$ is uniquely determined by the difference
\[\hat\nabla^{\mathcal V}-\nabla^{\mathcal V}\in\Omega^1(M,L(\mathcal V,\mathcal V)).\]
Note that the natural filtration in $\mathcal V$ and the natural filtration $T^*M=T^*M^{(1)}\supseteq T^*M^{(2)}\supseteq\cdots$ together induce a homogeneity filtration in $\Omega^1(M,L(\mathcal V,\mathcal V))$. Recall from \cite{HSSS12A}*{Section 1} that if $\hat\nabla^{\mathcal V}$ is a prolongation connection, then the homogeneiry of $\hat\nabla^{\mathcal V}-\nabla^{\mathcal V}$ is at least $1$. Conversely, it is proven that there exists a unique linear connection $\hat\nabla^{\mathcal V}$ such that $\hat\nabla^{\mathcal V}-\nabla^{\mathcal V}$ has homogeneity at least $1$ and for all $\mu\in\Gamma(\mathcal V)$, there holds
\begin{align*}
\hat\nabla^{\mathcal V}\mu-\nabla^{\mathcal V}\mu&\in\Gamma(\mathrm{im}\,\partial^*_{\mathcal V})\subseteq\Omega^1(M,\mathcal V)\quad\text{and}\\
R(\hat\nabla^{\mathcal V})\,\mu&\in\Gamma(\ker\,\partial^*_{\mathcal V})\subseteq\Omega^2(M,\mathcal V).
\end{align*}
Moreover, $\hat\nabla^{\mathcal V}$ is a prolongation of the first BGG operator. Here, $R(\,\cdot\,)$ denotes the curvature of a given linear connection and $\partial^*_{\mathcal V}$ is the Lie algebra boundary operator of $\Omega^\bullet(M,\mathcal V)$ as recalled in \eqref{codiff}. Additionally, the relation \cite{HSSS12A}*{Lemma 2.3} between $\hat\nabla^{\mathcal V}\mu-\nabla^{\mathcal V}\mu$ and $R(\hat\nabla^{\mathcal V})\,\mu-R(\nabla^{\mathcal V})\,\mu$ is closely analogous to the relation \cite{Book}*{Proposition 3.1.10} between the change of Cartan connection and the change of Cartan curvature on a given parabolic geometry. This leads the authors of cite{HSSS12A} to a general procedure \cite{HSSS12B}*{Subsection 1.4} for computing the normal prolongation connection. It has been stated in \cite{HSSS12B}*{Examples 7.1} that in the case $\mathcal V\in\{\mathcal T,\mathcal T^*\}$ of standard (co)tractor bundle of an AG-structure $(M,E,F)$ of type $(2,n)$, the normal prolongation connection is determined by
\begin{align*}
\hat\nabla^{\mathcal V}\mu-\nabla^{\mathcal V}\mu=-\square^{-1}_{\mathcal V}\,\partial^*_{\mathcal V}\, d^{\nabla^{\mathcal V}}\, d^{\nabla^{\mathcal V}}\,\mu.    
\end{align*}
Here, $d^{\nabla^{\mathcal V}}$ is the covariant exterior derivative of the tractor connection $\nabla^{\mathcal V}$, \[\square_{\mathcal V}=\partial_{\mathcal V}\,\partial^*_{\mathcal V}+\partial^*_{\mathcal V}\,\partial_{\mathcal V}\] is the Kostant Laplacian where $\partial_{\mathcal V}$ is the Lie algebra differential that is adjoint to $\partial^*_{\mathcal V}$; see, e.g., \cite{Book}*{Subsections 1.3.1 and 3.3.1}. Note that there holds
\[\square_{\mathcal V}=\partial^*_{\mathcal V}\,\partial_{\mathcal V}\]
on $\Omega^1(M,\mathcal V)$ and
\[d^{\nabla^{\mathcal V}}\, d^{\nabla^{\mathcal V}}\,\mu=R(\nabla^{\mathcal V})\,\mu=\kappa\cdot\mu\]
where $\kappa\in\Omega^2(M,L(\mathcal V,\mathcal V))$ is the normal Cartan curvature of the geometric structure applied to $\mathcal V$; see, e.g., \cite{Book}*{Corollary 1.5.7}.

Assume that a Weyl structure has been chosen. In the case where $\mathcal V=\mathcal T$ is the standard tractor bundle of $(M,E,F)$, the relevant (co)differentials are given by
\begin{align*}
\partial^*_{\mathcal T}:\Omega^2(M,F\oplus E)&\to \Omega^1(M,F\oplus E)\\
(\beta'{}^A_{A'}{}^B_{B'}{}^{C'},\beta{}^A_{A'}{}^B_{B'}{}^{C})&\mapsto(0^{B'},2\,\beta'{}^A_{A'}{}^B_{I'}{}^{I'})=\partial^*_{\mathcal T}(\beta',\beta){}^A_{A'}{}^{B'B}\quad\text{and}\\
\partial_{\mathcal T}:\Omega^1(M,F\oplus E)&\to \Omega^2(M,F\oplus E)\\
(\alpha'{}^A_{A'}{}^{B'},\alpha{}^A_{A'}{}^{B})&\mapsto(\delta{}^{C'}_{A'}\,\alpha{}^B_{B'}{}^A-\delta{}^{C'}_{B'}\,\alpha{}^A_{A'}{}^B,0^C)=\partial_{\mathcal T}(\alpha',\alpha){}^A_{A'}{}^B_{B'}{}^{C'C}.
\end{align*}
In particular, for all $\alpha{}^A_{A'}{}^B\in\Omega^1(M,E)$, there holds
\begin{align*}
(\square_{\mathcal T}\,\alpha){}^A_{A'}{}^B
=2\,(\delta{}^{I'}_{A'}\,\alpha{}^B_{I'}{}^A-\delta{}^{I'}_{I'}\,\alpha{}^A_{A'}{}^B)
=2\,\alpha{}^B_{A'}{}^A-2\,n \,\alpha{}^A_{A'}{}^B   
\end{align*}
and hence
\begin{align*}
(\square_{\mathcal T}^{-1}\,\alpha){}^A_{A'}{}^B=-\frac{1}{2\,n-2} \,\alpha{}^({}^{A}_{A'}{}^{B}{}^)-\frac{1}{2\,n+2} \,\alpha{}^[{}^{A}_{A'}{}^{B}{}^].  
\end{align*}
Inserting $(\eta^{A'},\xi^A)\in\Gamma(F\oplus E)\cong\Gamma(\mathcal T)$, we obtain the following relation.
\begin{align*}
d^{\nabla^{\mathcal T}}\, d^{\nabla^{\mathcal T}}\,(\eta,\xi){}^A_{A'}{}^B_{B'}{}^{C'C}
&=\kappa{}^A_{A'}{}^B_{B'}\,(\eta^{C'},\xi^C)=(\tau{}^{A}_{A'}{}^{B}_{B'}{}^{C'}_I\xi^I+W'{}^{A}_{A'}{}^{B}_{B'}{}^{C'}_{I'}\eta^{I'}, W{}^{A}_{A'}{}^{B}_{B'}{}^{C}_I\xi^I+Y{}^{A}_{A'}{}^{B}_{B'}{}^{C}_{I'}\eta^{I'}),\\
\partial^*_{\mathcal T}\,d^{\nabla^{\mathcal T}}\, d^{\nabla^{\mathcal T}}\,(\eta,\xi){}^A_{A'}{}^{B}&=2\,W'{}^{A}_{A'}{}^{B}_{J'}{}^{J'}_{I'}\,\eta^{I'}=2\,tr(W){}^{A}_{A'}{}^{B}_{I'}\,\eta^{I'},\\
\square_{\mathcal T}^{-1}\,\partial^*_{\mathcal T}\,d^{\nabla^{\mathcal T}}\, d^{\nabla^{\mathcal T}}\,(\eta,\xi){}^A_{A'}{}^{B}&=-\frac{1}{n-1} \,tr(W){}^{(A}_{(A'}{}^{B)}_{I')}\,\eta^{I'}-\frac{1}{n+1} \,
tr(W){}^{[A}_{[A'}{}^{B]}_{I']}\,\eta^{I'}=\Phi{}^A_{A'}{}^{B}_{I'}\,\eta^{I'};
\end{align*}
see \cref{Phi}. In particular, applying the above general procedure results in the prolongation connection defined in \cref{prolongation connection} (i).

Similarly, in the case where $\mathcal V=\mathcal T^*$ is the standard cotractor bundle of $(M,E,F)$, the relevant (co)differentials are given by
\begin{align*}
\partial^*_{\mathcal T^*}:\Omega^2(M,E^*\oplus F^*)&\to \Omega^1(M,E^*\oplus F^*)\\
(\beta{}^A_{A'}{}^B_{B'}{}_C,\beta'{}^A_{A'}{}^B_{B'}{}_{C'})&\mapsto(0_{B},-2\,\beta{}^A_{A'}{}^I_{B'}{}_{I})=\partial^*_{\mathcal T^*}(\beta,\beta'){}^A_{A'}{}_{BB'}\quad\text{and}\\
\partial_{\mathcal T^*}:\Omega^1(M,E^*\oplus F^*)&\to \Omega^2(M,E^*\oplus F^*)\\
(\alpha{}^A_{A'}{}_{B},\alpha'{}^A_{A'}{}_{B'})&\mapsto(-\delta{}^{A}_{C}\,\alpha'{}^B_{B'}{}_{A'}+\delta{}^{B}_{C}\,\alpha{}^A_{A'}{}_{B'},0_C')=\partial_{\mathcal T^*}(\alpha,\alpha'){}^A_{A'}{}^B_{B'}{}_{CC'}.
\end{align*}
In particular, for all $\alpha'{}^A_{A'}{}_{B'}\in\Omega^1(M,F^*)$, there holds
\begin{align*}
(\square_{\mathcal T^*}\,\alpha'){}^A_{A'}{}_{B'}=-2\,(-\delta{}^{A}_{I}\,\alpha'{}^I_{B'}{}_{A'}+\delta{}^{I}_{I}\,\alpha'{}^A_{A'}{}_{B'})
=2\,\alpha'{}^A_{B'}{}_{A'}-4 \,\alpha'{}^A_{A'}{}_{B'}   
\end{align*}
and hence
\begin{align*}
(\square_{\mathcal T^*}^{-1}\,\alpha'){}^A_{A'}{}_{B'}=-\frac{1}{2} \,\alpha{}_({}^{A}_{A'}{}_{B'}{}_)-\frac{1}{6} \,\alpha{}_[{}^{A}_{A'}{}_{B'}{}_].  
\end{align*}
Inserting $(\varphi_A,\mu_{A'})\in\Omega^1(E^*\oplus F^*)\cong\Gamma(\mathcal T^*)$, we obtain the following relation.
\begin{align*}
d^{\nabla^{\mathcal T^*}}\, d^{\nabla^{\mathcal T^*}}\,(\varphi,\mu){}^A_{A'}{}^B_{B'}{}_{CC'}
&=(-W{}^{A}_{A'}{}^{B}_{B'}{}^I_C\,\varphi_I-\tau{}^{A}_{A'}{}^{B}_{B'}{}^{I'}_C\,\mu_{I'},-Y{}^{A}_{A'}{}^{B}_{B'}{}^I_{C'}\,\varphi_I-W'{}^{A}_{A'}{}^{B}_{B'}{}^{I'}_{C'}\,\mu_{I'}),\\
\partial^*_{\mathcal T^*}\,d^{\nabla^{\mathcal T^*}}\, d^{\nabla^{\mathcal T^*}}\,(\varphi,\mu){}^A_{A'}{}_{B'}&=2\,W{}^{A}_{A'}{}^{J}_{B'}{}^{I}_{J}\,\varphi_I=2\,tr(W){}^{A}_{A'}{}^{I}_{B'}\,\varphi_I,\\
\square_{\mathcal T^*}^{-1}\,\partial^*_{\mathcal T^*}\,d^{\nabla^{\mathcal T^*}}\, d^{\nabla^{\mathcal T^*}}\,(\varphi,\mu){}^A_{A'}{}_{B'}&=-tr(W){}^{(A}_{(A'}{}^{I)}_{B')}\,\varphi_I-\frac{1}{3} \,
tr(W){}^{[A}_{[A'}{}^{I]}_{B']}\,\varphi_I=\Psi{}^A_{A'}{}^{I}_{B'}\,\varphi^{I'};
\end{align*}
see \cref{Phi}. In particular, applying the above general procedure results in the prolongation connection defined in \cref{prolongation connection} (ii).

The above computation has also been made in \cite{HSSS12B}*{Example 7.1}. Note that the contraction of the torsion quadratic they use corresponds to $-tr(\iota_\tau\tau){}^A_{A'}{}^B_{B'}$ and the contraction of Weyl tensor they use corresponds to $-tr(W){}^B_{A'}{}^A_{B'}=-tr(W){}^{(A}_{(A'}{}^{B)}_{B')}+tr(W){}^{[A}_{[A'}{}^{B]}_{B']}$. In addition, their convention on torsion, Rho tensor and abstract index notation also differ from ours. We point out that their result on the prolongation cotractor, i.e., the formula for $\tilde \nabla_c G_\beta$ in \cite{HSSS12B}*{Example 7.1} is misprinted and should read
\begin{align*}
\tilde\nabla_c G_\beta=\nabla_c G_\beta-X{}^\omega_{D'}X{}^D_\beta\left[S{}_({}^{D'}_C{}^{C'}_D{}_)
+\frac{1}{3}\,S{}_[{}^{D'}_C{}^{C'}_D{}_]\right]\,G_\omega
\end{align*}
where
\begin{align*}
S{}^{A'}_A{}^{B'}_B=U{}^{R'}_A{}^{A'}_B{}^{B'}_{R'}=U{}^{A'}_R{}^{B'}_A{}^{R}_{B}=\frac 1 q\, T{}^{\ \,(A'|e|}_{r\,(A}\,T{}^{B')\ \,r}_{B\,)\,e}-\frac{1}{q+4}\,T{}^{\ \,[A'|e|}_{r\,[A}\,T{}^{B']\ \,r}_{B\,]\,e}.
\end{align*}
\end{remark}
\begin{remark}\label[remark]{one jet}
By \eqref{general theory} in \cref{prolongation connection} we may identify the solution space of \[D:\Gamma(F)\to\Gamma(E\otimes\mathfrak{sl}(F))\] with the space of parallel sections of $\hat\nabla^{\mathcal T}$. Assume in addition that $M$ is connected. Then a parallel section of $\hat\nabla^{\mathcal T}$ is uniquely determined by its value in a point via the parallel transport. In particular, the space of parallel sections, i.e., the solution space of $D$ is a vector space of dimension at most $(n+2)$. Moreover, every solution of $D$ is completely determined by its 1-jet at a point. Indeed, if $\eta$ and $\tilde\eta$ are solutions of $D$ such that $J^1_x\,\eta=J^1_x\,\tilde\eta$ for some $x\in M$, then $L\,\eta$ agrees with $L\,\tilde\eta$ at $x$, hence everywhere in $M$. Hence $\eta=\tilde\eta$. 

Analogous conclusions hold for the first BGG operator $D:\Gamma(E^*)\to\Gamma(F^*\otimes\mathfrak{sl}(E))$ associated to the standard cotractor bundle $\mathcal T^*$.    
\end{remark}

\begin{remark}\label{curved orbit}
Recall the curved orbit theory for Cartan geometry: Consider a Cartan geometry $(\mathcal G\to M,\omega)$ of type $(G,P)$ with curvature $\kappa\in\Omega^2(\mathcal G,\mathfrak g)$ and a tractor bundle $\mathcal V=\mathcal G\times_P\mathbb V$ associated to a $G$-representation $\mathbb V$. Denote by $\mathbb V=\dot{\cup}\,\mathcal O_\alpha$ the partition of $\mathbb V$ into its $P$-orbits. Then any section $\mu$ of $\mathcal V$, viewed as a $P$-equivariant map $\mu:\mathcal G\to\mathbb V$, induces a $P$-invariant partition 
\[\mathcal G=\dot{\cup}\,\mu^{-1}(\mathcal O_\alpha)=\dot{\cup}\,\mathcal G_\alpha\] that projects to $M=\dot{\cup}\,M_\alpha$. Viewing $\mu$ as a map $M\to \mathcal V$, there holds
\[M_\alpha=\mu^{-1}(\mathcal G\times_P\mathcal O_\alpha).\]
Assume in addition that $\mu$ is parallel. Then every $M_\alpha$ is called the curved orbit associated to the $P$-orbit $\mathcal O_\alpha\subseteq\mathbb V$; see \cite{CGH}*{Definition 2.4} and the paragraph succeeding it. Choose a base point $v_\alpha\in\mathcal O_\alpha$ for each $P$-orbit and denote by 
\[G_\alpha=\{g\in G:g\,v_\alpha=v_\alpha\}\] the corresponding isotropy subgroup of $G$. Then $\omega|_{\mathcal G_\alpha}$ takes values in the Lie algebra $\mathfrak g_\alpha$ of $G_\alpha$ and hence
\[(\mathcal G_\alpha\to M_\alpha,\omega|_{\mathcal G_\alpha})\]
is a Cartan geometry of type $(G_\alpha,G_\alpha\cap P)$ with curvature $\kappa_\alpha=\kappa|_{\mathcal G_\alpha}\in\Omega^2(\mathcal G_\alpha,\mathfrak g_\alpha)$; see \cite{CGH}*{Theorem 2.6 (ii) and (iii)}. Moreover, the pair $M_\alpha\subseteq M$ is locally diffeomorphic to the pair $(G_\alpha\,P)/P\subseteq G/P$. Precisely, for every $x\in M_\alpha$ there exists a local diffeomorphism
\[M\supseteq U\xrightarrow[\varphi]{\cong}U'\subseteq G/P\]
such that $U$ is an open subset of $M$ containing $x$, $U'$ is an open subset of $G/P$ containing $eP$ and there holds $\varphi(U\cap M_\alpha)=U'\cap (G_\alpha\,P)/P$; see \cite{CGH}*{Theorem 2.6 (i)}. In particular, since $(G_\alpha\,P)/P\subseteq G/P$ is initial and hence immersive, $M_i\subseteq M$ is also initial and hence immersive; see \cite{CGH}*{Theorem 2.6 (i)}. If $(G_\alpha\,P)/P\subseteq G/P$ is in addition regular, i.e., embedded, then $M_i\subseteq M$ is also an embedded submanifold. In fact, this is the case of all examples in \cite{CGH}.

Since the $P$-orbit $\{0\}\subseteq\mathbb V$ has isotropy group $G$ and induces a curved orbit $\{\mu=0\}$ that is either empty or is the union of some connected components of $M$, we exclude this trivial situation by assuming that $M$ is connected and that $\mu$ is not the zero section.
\end{remark}
Consider now the standard tractor bundle $\mathcal T$ on an AG-structure $(M,E,F)$. In view of \cref{prolongation connection}, if $M$ is connected and $0\neq\eta\in\Gamma(F)$ is a solution of the BGG operator, then $L\,\eta\in\Gamma(\mathcal T)$ is parallel for the modified Cartan connection $\hat\omega$ and is nowhere-vanishing. The $P$-orbits of $\mathbb R^{n+2}\setminus\{0\}$ are precisely 
\begin{align*}
    \mathbb R^2\setminus\{0\}\quad\text{and}\quad
    \mathbb R^2\times(\mathbb R^n\setminus\{0\}).
\end{align*}
As recalled in \cref{curved orbit} above, applying $(L\,\eta)^{-1}(\mathcal G\times_P\,\cdot\,)$ to each $P$-orbit gives rise to the curved orbit decomposition of $M$ into
\begin{equation}\label{curved orbit spaces}
\begin{aligned}
\{x\in M:(L\,\eta)(x)\in E_x\setminus\{0\}\}=
    \{x\in M:\eta(x)&=0\}\quad\text{and}\\
    \{x\in M:\eta(x)&\neq 0\}.
\end{aligned}\end{equation}
The case for the standard cotractor bundle is analogous. While one could investigate regularity and the geometric structure of each curved orbit from the general theory recalled in \cref{curved orbit}, we provide an alternative argument in standard geometric terms in the rest of the article.

\begin{corollary}\label[corollary]{open dense}
Notation as in \cref{splitting operator}. Assume in addition that $M$ is connected.

\begin{enumerate}
    \item [(i)]If $0\neq \eta\in\Gamma(F)$ with $D\,\eta=0$, then $\{x\in M:\eta(x)\neq 0\}$ is an open dense subset of $M$.
\item [(ii)] If $0\neq \varphi\in\Gamma(E^*)$ with $D\,\varphi=0$, then $\{x\in M:\varphi(x)\neq 0\}$ is an open dense subset of $M$.
\end{enumerate}
In particular, these open submanifolds inherit the AG-structure of type $(2,n)$ from $(M,E,F)$.
\end{corollary}
\begin{proof}
We prove (i). Then (ii) follows in analogous steps. It is clear that $\{x\in M:\eta(x)\neq 0\}$ is open in $M$. Moreover, it follows from \cref{surjectivity} or, equivalently, \cref{one jet} that $J^1_x\eta\neq 0$ whenever $\eta(x)=0$. Hence $\{x\in M:\eta(x)\neq 0\}$ is dense in $M$.
\end{proof}

In addition to \cref{splitting operator} and \cref{normal solutions}, we provide a characterization of all nowhere vanishing solutions to the first BGG operator $D$.
\begin{proposition}\label{1st BGG operator}
A nowhere vanishing section of $F$, respectively, of $E^*$, is a solution to $D(\,\cdot\,)=0$ where $D$ is the first BGG operator of $\mathcal T$, respectively, of $\mathcal T^*$, if and only if it is parallel for some Weyl connection.
\end{proposition}
\begin{proof}
Let $\eta\in\Gamma(F)$. With respect to the decomposition $T^*M\otimes F=E\otimes\mathfrak{sl}(F)\oplus E$, the formula
$\hat\nabla^A_{A'}\eta^{B'}=\nabla^A_{A'}\eta^{B'}+\Upsilon^A_{I'}\,\eta^{I'}\,\delta^{B'}_{A'}$ in the beginning of the section decomposes to
\begin{align}\label{nabla eta}
    (\hat\nabla\eta)_o=(\nabla\eta)_o\quad\text{and}\quad\frac 1 n\,\hat\nabla^A_{I'}\eta^{I'}=\frac 1 n\,\nabla^A_{I'}\eta^{I'}+\Upsilon^A_{I'}\,\eta^{I'}.
\end{align}
In particular, 
\[\{\Upsilon\in\Omega^1(M):\Upsilon(\eta)=-\frac 1 n\,tr(\nabla\eta)\}\]
determines a collection of Weyl structures $\hat\sigma$ whose Weyl connections $\hat\nabla$ has the property that $\hat\nabla\eta=D\,\eta$. Similarly, for every $\varphi\in\Gamma(E^*)$ there is a collection of Weyl structures with Weyl connections $\hat\nabla$ such that $\hat\nabla\varphi=D\,\varphi$.
\end{proof}

Now we move on to study the other space defined in \eqref{curved orbit spaces}, namely, the zero locus $\{x\in M:\eta(x)=0\}$ of a nonzero solution of $D$ and its cotractor analog. First, we develop in \cref{level set} a vector bundle version of the regular level set theorem. Precisely, we consider any vector bundle $V\to M$ on any manifold $M$ with no additional assumption and characterize every section $\sigma\in\Gamma(V)$ that is a defining section for embedded submanifolds of $M$; see, e.g., \cite{CGH21}*{Definition 2.16}.

\begin{lemma}\label[lemma]{level set}
    Let $V\to M$ be a vector bundle of rank $k$ on a manifold $M$ and $\nabla$ a linear connection on $V$. If $\sigma\in\Gamma(V)$ is a section such that
    \[N=\{x\in M:\sigma(x)=0\}\]
    is nonempty and
    \[\nabla\sigma(x):T_xM\to V_x\]
    is surjective for every $x\in N$, then $N$ is a regular, i.e., embedded, submanifold of $M$ with codimension $k$. Moreover, there holds
    \[T_xN=\ker(\nabla\sigma(x))\subseteq T_xM\]
    for every $x\in N$.   
\end{lemma}

\begin{proof}
Observe that $N\subseteq M$ is an embedded submanifold if and only if there exists a covering $\mathcal U$ of $M$ consisting of open subsets of $M$ such that $N\cap U\subseteq U$ is an embedded submanifold for every $U\in\mathcal U$. Thus we may assume without loss of generality that $V$ admits a global frame $s_1,\dots,s_k\in\Gamma(V)$. We write $\sigma=\sum_{i=1}^k\sigma^i\,s_i$ where $\sigma^i\in C^\infty(M)$. For $\xi\in\mathfrak X(M)$, there holds
\begin{align*}
\nabla_\xi\sigma=\sum_{i=1}^k\left(d\,\sigma^i(\xi)\,s_i+\sigma^i\nabla_\xi s_i\right).
\end{align*}
Observe that $x\in N$ if and only if $x$ is in the zero locus of $(\sigma^1,\dots,\sigma^k)\in C^\infty(M,\mathbb R^k)$. In this case, $\nabla\sigma(x)=\sum_{i=1}^k d\sigma^i(x)\,s_i(x)$ is assumed to be surjective and thus $(\sigma^1,\dots,\sigma^n)$ is a submersion at the point $x$. In particular, it follows from the regular level set theorem given in, e.g., \cite{Lee}*{Corollary 5.14} that $N$ is an embedded submanifold of $M$ with codimension $k$. Moreover, it follows from \cite{Lee}*{Proposition 5.37} that for $x\in N$, the joint kernel of $\{d\,\sigma^i|_x:i=1,\dots,k\}$ lies in $T_xN$. That is, $\ker(\nabla\sigma(x))\subseteq T_xN$. Counting dimensions, we conclude that $\ker(\nabla\sigma(x))= T_xN$. This completes the proof of the lemma.
\end{proof}
We also recall the following geometric fact.
\begin{lemma}\label[lemma]{restriction}
Let $V\to M$ be a vector bundle on a manifold $M$. If $N\subseteq M$ is either an immersed submanifold or an embedded submanifold, then
\begin{align*}
    V|_N=\bigcup_{x\in N}V_x
\end{align*}
is a vector bundle on $N$. Moreover, every linear connection $\nabla^V$ on $V$ restricts to a linear connection $\nabla^{V|_N}$ on $V|_N$.
\end{lemma}

\begin{proof}
That $V|_N$ is a vector bundle is precisely the statement of \cite{Lee}*{Example 10.8}. Now let $s_1,s_2\in\Gamma(V)$ be two sections such that $s_1=s_2$ on $N$. Clearly, $\nabla^{V}s_i$ induce the same one-form in $\Omega^1(N,V|_N)$, $i=1,2$.
\end{proof}

Now we are ready to prove that the zero locus of a BGG solution is an embedded submanifold and describe its tangent bundle.

\begin{theorem}[Regularity of the zero loci]\label[theorem]{zero locus}
Let $(M,E,F)$ be an AG-structure of type $(2,n)$ where the base space $M$ is connected.

\begin{enumerate}
\item [(i)] 
Assume that $0\neq\eta\in\Gamma(F)$ is a solution of $D$ such that 
\[N=\{x\in M:\eta(x)=0\}\]
is nonempty. Then $N$ is an $n$-dimensional embedded submanifold of $M$. The restriction $(L\,\eta)|_N\in\Gamma(E|_N)$ is nowhere vanishing with annihilator 
\[\ell=\ker\,((L\,\eta)|_N)\subseteq E^*|_N\]
of rank $1$. Moreover, the defining isomorphism $TM\cong E^*\otimes F$ of the AG-structure $(M,E,F)$ restricts to
\[TN\cong \ell\otimes F|_N.\]

\item [(ii)] Assume that $0\neq\varphi\in\Gamma(E^*)$ is a solution of $D$ such that 
\[\tilde N=\{x\in M:\varphi(x)=0\}\]
is nonempty. Then $\tilde N$ is a $2\,(n-1)$-dimensional embedded submanifold of $M$. The restriction $(L\,\varphi)|_{\tilde N}\in\Gamma(F^*|_N)$ is nowhere vanishing with kernel \[\tilde F=\ker\,((L\,\varphi)|_{\tilde N})\subseteq F|_N\]
of rank $(n-1)$. Moreover, the defining isomorphism $TM\cong E^*\otimes F$ of the AG-structure $(M,E,F)$ restricts to
\[T\tilde N\cong E^*|_{\tilde N}\otimes\tilde F.\]
\end{enumerate}

\end{theorem}

\begin{proof}Let $\eta$ satisfy the assumptions in (i). Recall from \cref{prolongation connection} (i) that $D\,\eta=0$ if any only if $L\,\eta$ is parallel with respect to the connection $\hat\nabla^{\mathcal T}$. Since $M$ is connected and $\eta\neq 0$, $(L\,\eta)^{A'A}=(\eta^{A'},-\frac 1 n \nabla{}^A_{I'}\eta^{I'})$ is nowhere vanishing. In particular, for all $x\in N$, $(L\,\eta)(x)$ is a nonzero element in $E_x$ with $1$-dimensional annihilator $\ell_x\in E^*_x$. Hence there holds
\[\ker(\nabla\eta(x))=\ell_x\otimes F_x;\]
see \cref{surjectivity}. Thus the statement of (i) follows from \cref{level set} and \cref{restriction}.

(ii) is proved in analogous steps. 
\end{proof}

Moreover, we describe the geometric structures on $N$ and $\tilde N$. Recall that a projective structure on a manifold $N$ is a projectively equivalence class of linear connections on $TN$. Two connections $\nabla,\hat\nabla$ on $TN$ are said to be projectively equivalent if there exists $\Upsilon\in\Omega^1(N)$ such that $\hat\nabla_{\xi_1}\xi_2-\nabla_{\xi_1}\xi_2=\Upsilon(\xi_1)\xi_2+\Upsilon(\xi_2)\xi_1$ for all $\xi_1,\xi_2\in\mathfrak X(N)$. It follows immediately that connections in the same projectively equivalence class have the same torsion, called the torsion of the projective structure.

\begin{theorem}[Geometric structures on the zero loci]\label[theorem]{canonical connections}
    Notation as in \cref{zero locus}. Denote by $\tau$ the torsion of $(M,E,F)$.
    \begin{enumerate}
        \item [(i)] Every Weyl connection $(\nabla^E,\nabla^F)$ of $(M,E,F)$ canonically induces a connection on $TN$. These connections on $TN$ define a projective structure on $N$. The torsion of this projective structure is given by the restriction of $\tau$ to $\Omega^2(N,TN)$.
        \item [(ii)] The nowhere vanishing section $(L\,\varphi)|_{\tilde N}\in\Gamma(F^*|_{\tilde N})$ canonically induces an isomorphism $\wedge^2\,E^*|_{\tilde N}\cong\wedge^{n-1}\,\tilde F$, up to a constant multiple. This isomorphism together with the isomorphism $T\tilde N\cong E^*|_{\tilde N}\otimes\tilde F$ obtained in \cref{zero locus} (ii) define an AG-structure of type $(2,n-1)$ on $(\tilde N,E|_{\tilde N},\tilde F)$.
    \end{enumerate}
\end{theorem}
\begin{remark}
From the statement of \cref{canonical connections} we see that rescaling a solution of one of the two first BGG operators by a nonzero constant does not change the geometric structure induced on its zero locus. Note that if two AG-structures differ only by a nonzero constant multiple on the defining line bundle isomorphism $\wedge^2\,E^*\cong \wedge^n\,F$, then their corresponding normal parabolic geometries are isomorphic.
\end{remark}
\begin{proof}[Proof of \cref{canonical connections}]
We first prove (i). Let $\nabla$ be a Weyl connection of $(M,E,F)$ and $s=(L\,\eta)|_N\in\Gamma(E|_N)$. It follows from \cref{prolongation connection} (i) and the formula of $\nabla^{\mathcal T}$ in the beginning of the section that
\[0=\hat\nabla^{\mathcal T|_N}s=\nabla^{\mathcal T|_N}s=\nabla^{E|_N}s.\]
Here, $\nabla^{E|_N}$ is the restriction of the Weyl connection $\nabla$; see \cref{restriction}. In particular, if $\xi\in\mathfrak X(N)$ and $\alpha\in\Gamma(E^*|_N)$ such that $\alpha(s)=0$, then $(\nabla^{E^*|_N}_\xi\alpha)(s)=0$. That is, $\nabla^{E^*|_N}$ preserves the subbundle $\ell\subseteq E^*|_N$. Together with the restriction of $\nabla$ to $F|_N$ we obtain a connection on $TN\cong\ell\otimes (F|_N)$. Let $\hat\nabla$ be the Weyl connection of $(M,E,F)$ determined by $\Upsilon\in\Omega^1(M)$; see the beginning of the section. Using the formulae in the beginning of the section, one computes that
\begin{align*}
\hat\nabla^{\ell}_{\alpha\otimes\beta}\alpha=\nabla^{\ell}_{\alpha\otimes\beta}\alpha+\Upsilon(\alpha\otimes\beta)\alpha\quad\text{and}\quad
\hat\nabla^{F|_N}_{\alpha\otimes\beta}\tilde\beta=\nabla^{F|_N}_{\alpha\otimes\beta}\tilde\beta+\Upsilon(\alpha\otimes\tilde\beta)\beta
\end{align*}
for all $\alpha\in\Gamma(\ell)$ and $\beta,\tilde\beta\in\Gamma(F|_N)$. In particular, the connection on $TN$ induced by $\nabla$ and $\hat\nabla$ are projectively equivalent. By construction, the connection on $TN$ has torsion $\tau|_{\wedge^2\,TN}$. This proves (i).

Now we turn to (ii). Consider the nowhere vanishing section $(L\,\varphi)|_{\tilde N}\in\Gamma(F^*|_{\tilde N})$. Viewing $F|_{\tilde N}=F^{**}|_{\tilde N}$, we obtain its interior product
\begin{align*}
\iota_{(L\,\varphi)|_{\tilde N}}:\wedge^n\,F|_{\tilde N}\xrightarrow{\cong}\wedge^{n-1}\,\tilde F\subseteq\wedge^{n-1}\,F|_{\tilde N}
\end{align*}
which defines an isomorphism onto its image. In particular, this isomorphism together with the defining isomorphism $\wedge^2\,E^*\cong\wedge^n\,F$ of $(M,E,F)$ give rise to
\[\wedge^2\,E^*|_{\tilde N}\cong\wedge^{n-1}\,\tilde F.\]
Thus $(\tilde N,E|_{\tilde N},\tilde F)$ is an AG-structure of type $(2,n-1)$. This proves (ii).
\end{proof}

We refer the reader to the recent papers \cite{FG19} and \cite{FG23} as well as the references therein for related results for different geometric structures.


\end{document}